\documentclass[11pt,a4paper]{amsart}

\usepackage{geometry}
\pdfoutput=1
\usepackage{amsmath,amsfonts,amssymb,amsthm}
\usepackage{xcolor}
\usepackage{todonotes}
\usepackage{hyperref}
\hypersetup{
	pagebackref  = true,
	colorlinks   = true,
	pdfauthor    = {Vitezslav Kala and Pavlo Yatsyna},
	pdftitle     = {Lifting problem for universal quadratic forms},
	pdfkeywords  = {},
	pdfcreator   = {Pdflatex},
	pdfpagemode  = UseNone,
	pdfstartview = FitH
}
\usepackage{pdfpages}
\usepackage[normalem]{ulem}
\usepackage{url}
\usepackage{enumerate}
\usepackage{verbatim}
\usepackage{mathrsfs}
\usepackage{csquotes}
\usepackage{mathtools}
\usepackage{lipsum}
\numberwithin{equation}{section}

\newcommand\QQ{\mathbb{Q}}

\newcommand\ZZ{\mathbb{Z}}

\theoremstyle{plain}
\newtheorem{theorem}{Theorem}
\numberwithin{theorem}{section}
\newtheorem{lemma}[theorem]{Lemma}

\newtheorem{cor}[theorem]{Corollary}
\newtheorem{prop}[theorem]{Proposition}

\newtheorem*{lemma*}{Lemma}
\theoremstyle{definition}

\newtheorem{definition}[theorem]{Definition}

\newcommand{\R}{\mathbb{R}}
\newcommand{\N}{\mathbb{N}}
\newcommand{\Z}{\mathbb{Z}}
\newcommand{\Q}{\mathbb{Q}}

\newcommand{\co}{\mathcal{O}}

\usepackage{tikz}

\DeclareMathOperator\Tr{Tr}

\newcommand*\py{\mathrm{P}}

\newcommand*\defined[1]{\textsl{#1}}

\DeclarePairedDelimiterX\set[2]\lbrace\rbrace{\,#1\mathclose{}:\mathopen{}#2\,}

\title[Lifting problem for universal quadratic forms]{Lifting problem for universal quadratic forms}

\author{V\' \i t\v ezslav Kala}
\address{Charles University, Faculty of Mathematics and Physics, Department of Algebra, Sokolov\-sk\' a 83, 18600 Praha~8, Czech Republic}
\email{vita.kala@gmail.com}

\author{Pavlo Yatsyna}
\address{Charles University, Faculty of Mathematics and Physics, Department of Algebra, Sokolov\-sk\' a 83, 18600 Praha~8, Czech Republic}
\email{pvyatsyna@gmail.com}

\thanks{Both authors were supported by the project PRIMUS/20/SCI/002 from Charles University.
	The first author was supported by Czech Science Foundation GA\v CR, grant 17-04703Y, and by Charles University Research Centre program UNCE/SCI/022. Most of the work on this article was done when the second author was visiting Charles University, also supported by Czech Science Foundation GA\v CR, grant 17-04703Y}
\date{\today}

\begin{document}
	
\maketitle

\begin{abstract}
	We study totally real number fields that admit a universal quadratic form whose coefficients are rational integers. We show that $\Q(\sqrt 5)$ is the only such real quadratic field, and that among fields of degrees 3, 4, 5, and 7 which have principal codifferent ideal, the only one is $\QQ(\zeta_7+\zeta_7^{-1})$, over which the form $x^2+y^2+z^2+w^2+xy+xz+xw$ is universal.
	Moreover, we prove an upper bound for Pythagoras numbers of orders in number fields that depends only on the degree of the number field.
\end{abstract}

\section{Introduction}\label{s:1}

The question which integers can be represented by a given quadratic form has long played a central role in number theory, involving works of mathematicians such as Diophantus, Brahmagupta, Fermat, Euler, and Gauss. Of particular interest have been universal quadratic forms, i.e., positive definite forms that represent all natural numbers.
The first example of the sum of four squares $x^2+y^2+z^2+w^2$ was followed by many others, including classification of quaternary diagonal universal forms by Ramanujan and Dickson, and culminating in the 15- and 290- theorems of Conway-Schneeberger and Bhargava-Hanke \cite{Bh, BH}.

A natural generalization has been the study of universal quadratic forms over number fields $K$ and their rings of algebraic integers $\co_K$. 
When the field has a complex embedding, every quadratic form over $K$ is indefinite, and so it is comparatively easy to understand which algebraic integers it represents. For example, Siegel \cite{Si3} and Estes-Hsia \cite{EH} considered complex fields with universal sums of 5 and 3 squares (respectively) and characterized them.
Hence of particular interest are totally real number fields where one expects to have a rich and hard theory of representations by totally positive quadratic forms.

In 1941 Maa\ss \ \cite{Ma} used theta series to show that the sum of three squares is universal over the ring of integers of $\Q(\sqrt 5)$.
Siegel \cite{Si3} then in 1945 proved that the sum of any number of squares is universal only over the number fields  $K=\Q, \Q(\sqrt 5)$. 
However, universal forms exist over every totally real number field \cite{HKK}, and there have been numerous recent results concerning them, see, e.g., \cite{CKR, EK, Ea, Ki, Ki2, CI, Sa, De, BK, BK2, Ka, Ya, KS, CL+, KTZ} and the references therein.

Almost all of these results involve quadratic forms which do not have rational integers as all of their coefficients. 
This is not an accident, as indeed Siegel's result immediately implies that
a diagonal positive definite quadratic form with $\Z$-coefficients can be universal only over $K=\Q, \Q(\sqrt 5)$. This suggests the following natural generalization: when is it possible for a positive definite quadratic form with $\Z$-coefficients to be universal over the ring of integers $\co_K$ of a number field $K$? Or more generally, one can consider two (totally real) number fields $K\subset L$ and ask whether there is a quadratic form with $\co_K$ coefficients that is universal over $\co_L$. This is sometimes known as the lifting problem for universal quadratic forms over number fields; the main goal of the present article is to consider it for \emph{$\Z$-forms}, i.e., positive definite quadratic form with $\Z$-coefficients.

We completely solve this problem for real quadratic fields by proving

\begin{theorem}\label{Thm:1}
	There does not exist a $\Z$-form that is universal over a real quadratic number field $K$, unless $K=\Q(\sqrt{5}).$
\end{theorem}

Over $\Q(\sqrt 5)$, there are indeed quite a few universal $\Z$-forms, such as $x^2+y^2+z^2$ \cite{Ma}, $x^2+y^2+2z^2$ \cite{CKR}, $x^2+xy+y^2+z^2+zw+w^2$ \cite{De}, $x^2+y^2+z^2+w^2+xy+xz+xw$ \cite{De2}.
Lee \cite{Le} classified all quaternary classical universal forms (recall that a form is classical if all its off-diagonal coefficients are divisible by 2), but his list does not give any other examples that are $\Z$-forms. We are not aware of any classification of universal $\Z$-forms over $\Q(\sqrt 5)$ (or any other number field); this is another very interesting open problem.

We then turn our attention to number fields of higher degree, where the situation is much more convoluted.

In the spirit of the study of the minimal number of variables required by a universal form, we first show in Corollary \ref{cor:univ class} that there are no classical universal $\Z$-forms of rank strictly less than 6 over any totally real number field (of arbitrary degree), except for $\Q, \Q(\sqrt 5)$.

We finally focus on the existence of universal $\Z$-forms over certain number fields of small degree.

\begin{theorem}\label{Thm:2}
	There does not exist a totally real number field $K$ of degree $1,2,3,4,5$ or $7$ which has principal codifferent ideal and a universal $\Z$-form defined over it, unless $K=\Q, \Q(\sqrt{5})$ or $\QQ(\zeta_7+\zeta_7^{-1})$. 
	
	The $\Z$-form $x^2+y^2+z^2+w^2+xy+xz+xw$ is universal over $\QQ(\zeta_7+\zeta_7^{-1})$.
\end{theorem}

We prove this result as Theorems \ref{thm 16} and \ref{prop:7univ}. Note that the codifferent is principal for example when $\co_K=\Z[\alpha]$ for some $\alpha$ or when $K$ has class number one.

The limiting assumptions in the theorem come from the tools that we use. First of all, the composition of a $\Z$-form with (twisted) trace form decomposes as a tensor product, and so
we study tensor products of positive definite $\Z$-lattices and their minimal vectors. In particular, lattices of $E$-type (which were first introduced by Kitaoka \cite{Kt3}) play a prominent role in Section \ref{s:4}. We use them to show that if certain additively indecomposable algebraic integers are represented by a $\Z$-form, then they have to be squares. This in turn for example implies that if a number field possesses a universal $\Z$-form, then it has units of all signatures.
Not every lattice is of $E$-type, but in the small degrees considered in Theorem \ref{Thm:2} this poses no restrictions.

Our second main tool is Siegel's formula for the value of Dedekind zeta function at $s=-1$ \cite{Si1, Z}, which expresses this value in terms of elements of the codifferent of small trace. In particular, we are interested in elements of trace 1, and these are the only ones that appear in the formula for degrees $2,3,4,5$ or $7$. When the codifferent is principal, the resulting bound on their number gives
(together with the results of Section \ref{s:4}) an estimate on the number of minimal vectors of the trace form. However, this estimate can hold only for very few number fields, which in turn implies the theorem.
It is tempting to try to apply Siegel's formula also for higher degrees by (for example) using elements of trace 2 to deduce the existence of elements of trace 1. Unfortunately, the resulting bounds seem to be too weak to be of much use.

In several of the proofs we use computer calculations to deal with specific number fields and quadratic forms. All of these computations were done in Magma \cite{Magma} and are straightforward.

Note that Theorem \ref{Thm:1} is a special case of Theorem \ref{Thm:2}. However, in the quadratic case we have more explicit control of elements of small trace, and so the proof of Theorem \ref{Thm:1} is more elementary and does not require the use of Dedekind zeta function.

\

The question whether there exists a universal $\Z$-form over a number field not covered by Theorem \ref{Thm:2} remains open and may be very hard. Our results provide some clues towards conjecturing that there are perhaps no number fields with $\Z$-forms except for $\Q, \Q(\sqrt{5})$, and $\QQ(\zeta_7+\zeta_7^{-1})$, but the evidence is of course quite weak. Even more broadly, the following general lifting problem question remains completely open.

\textbf{Question.} \textit{Is there a totally real number field $K$ such that there are infinitely many totally real number fields $L\supset K$ that admit a universal quadratic form with $\co_K$-coefficients?}

\

The auxiliary results that we obtain are also useful for the study of Pythagoras numbers of orders $\co$ in totally real fields. While we know that typically not all totally positive integers are sums of squares, we can ask what is the smallest integer $m$ such that if an element is the sum of squares, then it is the sum of at most $m$ squares.
This integer $m$ is called the Pythagoras number of the order $\co$ and is known to be always finite, but can be arbitrarily large \cite{Sc2} (cf. also \cite{Po}).
In the aforementioned article, Scharlau asked whether Pythagoras numbers of orders are bounded by the degree of the corresponding number field. We answer this question affirmatively as Corollary \ref{cor:pythagoras}.

Let us note that while we state most of our results only for the maximal order $\co_K$ (as it is, arguably, the most interesting case), many of them can probably be quite straightforwardly extended 
to general orders, at least the material in Sections \ref{s:2}, \ref{s:3}, \ref{s:4}, and \ref{s:6}. However, the references that we use typically also deal only with the full ring of integers $\co_K$, and so this extension would require reproving all these referenced results in the more general setting.

\

Finally, let us note that besides from studying representations of integers by quadratic forms, there have been numerous works considering representations of quadratic forms by quadratic forms and, in particular, by the sum of squares, e.g., \cite{BI, Ic, KO, Sa1, Oh, KO2, KO3, JKO, BC+} (in fact, we use some of them in the proof of Corollary \ref{cor:pythagoras}). Most of them deal with forms over $\Z$, but it is another exciting direction of future research to consider the situation over number fields in detail and to study the lifting problem for representations of quadratic forms of a given rank $>1$.

\section*{Acknowledgments}

We thank Tom\' a\v s Hejda for our interesting conversations and for his help with computations related to this research, and to the anonymous referees for numerous useful comments that have helped to improve the paper and its exposition.

\section{Preliminaries}\label{s:2}

Throughout the article, $K$ will denote a totally real number field of degree $d$ over $\mathbb Q$ with the ring of integers $\co_K$. For an order $\co\subset\co_K$ we fix an integral basis $\omega_1, \dots, \omega_d$ and denote its group of units by $\co^\times$.

Let $\sigma_1=\mathrm{id}, \sigma_2, \dots, \sigma_d: K\hookrightarrow\mathbb R$ be the (distinct) real embeddings of $K$.
The norm of $\alpha\in K$ is then $\mathrm{N}(\alpha)=\sigma_1(\alpha)\cdots \sigma_d(\alpha)$ and its trace is $\Tr(\alpha)=\sigma_1(\alpha)+\dots+ \sigma_d(\alpha)$.

We write $\alpha \succ \beta$ to mean $\sigma_i(\alpha)>\sigma_i(\beta)$ for all $1\leq i\leq d$; moreover, $\alpha\succeq \beta$ denotes $\alpha\succ \beta$ or $\alpha=\beta$. An algebraic integer $\alpha\in\co_K$ is \defined{totally positive} if $\alpha\succ 0$; the semiring of totally positive integers that lie in the order $\co$ will be denoted $\co^+$; moreover we let
$\co^{\times,+}=\co^{\times}\cap\co^{+}$.
By the \defined{signature} of $\alpha\in K$ we mean the $d$-tuple of signs of $\sigma_i(\alpha)$.

We say that $\alpha\in\co^+$ is \defined{indecomposable} if it can not be decomposed as the sum of two totally positive elements of $\co$, or equivalently if there is no $\beta\in\co^+$ such that $\alpha\succ\beta$. Indecomposable integers and their norms are quite well studied, especially over real quadratic fields \cite{DS, JK, Ka2}.

The following lemma shows that when an element has sufficiently small norm, then it has to be indecomposable. In particular, every totally positive unit is indecomposable.

\begin{lemma}\label{lemma:norm_bound}
	a) For all $\alpha_1,\alpha_2\in \co^+$ we have \[\mathrm{N}(\alpha_1+\alpha_2)^{1/d}\ge \mathrm{N}(\alpha_1)^{1/d}+\mathrm{N}(\alpha_2)^{1/d}.\]
	
	b) If $\beta\in\co^+$ has norm $\mathrm{N}(\beta)<2^d$, then $\beta$ is indecomposable.
\end{lemma}

\begin{proof}
	Both parts are easy to show and quite well-known, but let us include the proofs for reader's convenience:
	
	a) follows by a simple use of H{\"o}lder's inequality \cite[3.1]{O2}: By (a version of) generalized H\" older's inequality \cite[Theorem 12]{HLP} we have 
	$$\prod_{i=1}^r\left(\sum_{j=1}^s a_{ij}^r\right)^{1/r}\geq\sum_{j=1}^s\prod_{i=1}^r a_{ij}, \mathrm{\ where\ } r,s\in\Z^+ \mathrm{\ and\ } a_{ij}\in\R^+.$$
	Setting $r=d, s=2$, and $a_{ij}=\sigma_i(\alpha_j)^{1/d}$ we get
	$$\mathrm{N}(\alpha_1+\alpha_2)^{1/d}=\prod_{i=1}^d\left(\sigma_i(\alpha_1+\alpha_2)\right)^{1/d}\geq
	\prod_{i=1}^d \sigma_i(\alpha_1)^{1/d}+\prod_{i=1}^d \sigma_i(\alpha_2)^{1/d}=
	\mathrm{N}(\alpha_1)^{1/d}+\mathrm{N}(\alpha_2)^{1/d}.$$	
	
	b) Assume that $\beta$ is decomposable as $\alpha_1+\alpha_2$. Then $2>\mathrm{N}(\beta)^{1/d}=\mathrm{N}(\alpha_1+\alpha_2)^{1/d}\ge \mathrm{N}(\alpha_1)^{1/d}+\mathrm{N}(\alpha_2)^{1/d}\geq 1+1,$
	which is not possible.
\end{proof}

We denote $\co^{\vee}:=\{\beta\in K:\Tr(\beta\co)\subset\ZZ\}$ the \defined{codifferent} of $\co$; $\co^{\vee,+}$ is the semiring of all totally positive elements of $\co^{\vee}$.
Recall that if $\co_K=\Z[\omega]$ for some $\omega$ with minimal polynomial $f(x)\in\Z[x]$, then the codifferent is the principal fractional ideal $\co_K^{\vee}=\frac 1{f'(\omega)}\co_K$ 
\cite[Proposition 4.17]{N}.

\medskip

We will often work with positive definite quadratic forms $Q\in\Z[x_1,\ldots,x_r]$, i.e., $Q(x)=\sum_{i\ge j}a_{ij}x_ix_j$ with $a_{ij}\in\Z$, and refer to such quadratic forms as \defined{$\Z$-forms} from now on; $r$ is the arity (or rank) of $Q$. Recall that $Q$ is \defined{positive semidefinite} if $Q(x)\geq 0$ for all $x\in\R^r$, and \defined{positive definite} if moreover $Q(x)=0$ implies $x=0$.
If $a_{ij}\in 2\Z$ for all $i\not=j$, then we say that $Q$ is \defined{classical} $\Z$-form (and non-classical, otherwise). 

For a given quadratic form $Q$ we can define the bilinear form $B_Q\in\Z[x_1,\dots,x_r]$ by $B_Q(x,y)=(Q(x+y)-Q(x)-Q(y))/2$. Note that if $e^{(1)}=(1,0,\dots,0),e^{(2)}=(0,1,0,\dots,0), \dots, e^{(r)}=(0,\dots,0,1)$, then $B_Q(e^{(i)},e^{(i)})=a_{ii}$ and $B_Q(e^{(i)},e^{(j)})=a_{ij}/2$ for $i\neq j$; the $r\times r$ matrix $M_Q:=(B_Q(e^{(i)},e^{(j)}))$ is called the \defined{Gram matrix} of $Q$.
 We have $Q(x)=x^T M_Q x$ and $B_Q(x,y)=x^T M_Q y$ for $x,y\in\Z^r$ 
($A^T$ denotes the transpose of a matrix $A$).

Unless specified otherwise, throughout the paper $Q$ will denote a $\Z$-form of rank $r$. 

\medskip

We will sometimes also need to work with quadratic forms over $\co$, i.e., $Q(x)=\sum_{i\ge j}a_{ij}x_ix_j$ with $a_{ij}\in\co$. The corresponding bilinear form $B_Q$ and Gram matrix $M_Q$ are defined in the same way as over $\Z$; we have $Q(x)=x^T M_Q x$ and $B_Q(x,y)=x^T M_Q y$ for $x,y\in\co^r$.
Such a form $Q$ is \emph{totally positive}  if $Q(a)\succ 0$ for all $a\in\co^r, a\neq 0$.

A $\Z$-form $Q$ can be naturally viewed also as a totally positive quadratic form $Q^{(\co)}$ over $\co$ with the same Gram matrix $M_{Q^{(\co)}}=M_Q$; although we will not explicitly distinguish between the two, as it should lead to no confusion.
Thus we will say that a $\Z$-form $Q$ represents an element $\alpha\in\co^+$ over the order $\co$ if  $Q(v)=\alpha$ for some $v\in \co^r$. We say that $Q$ is \defined{universal over $\co$} if it represents every element  $\alpha\in\co^+$ over $\co$. When dealing with the maximal order $\co_K$, we often just say that $Q$ is universal (or universal over $K$ to specify the number field).

\medskip

Sometimes it is useful to work in lattice-theoretic language, which we now briefly introduce.
A \defined{totally positive quadratic space over $K$} is an $r$-dimensional vector space $V$ over the field $K$ equipped with a symmetric bilinear form $B: V\times V\to V$ such that the associated quadratic form $Q(v)=B(v,v)$ attains totally positive values at all non-zero $v\in V$. A \defined{quadratic $\co$-lattice} $L\subset V$ is an $\co$-submodule such that $KL=V$; $L$ is equipped with the restricted quadratic form $Q$, and so we often talk about a quadratic $\co$-lattice as the pair $(L,Q)$.
A \defined{sublattice} of the lattice $(L,Q)$ is an $\co$-submodule equipped with the restriction of the quadratic form $Q$.
Note that to a totally positive quadratic form $Q$ over $\co$ we can naturally associate the
quadratic $\co$-lattice $(\co^r,Q)$, and so we will interchangeably talk of $Q$ as a quadratic form and as an $\co$-lattice. 

\medskip

We will use Siegel's result on the non-universality of sums of squares in the paper, so let us state it here formally as a theorem.

\begin{theorem}\label{thm:siegel}\cite[Theorem 1]{Si3}
	Let $K$ be a totally real number field different from $\Q$ and $\Q(\sqrt 5)$. Then some totally positive integer $\alpha\in\co_K^+$ is not the sum of squares of elements of $\co_K$. 
	In particular, the quadratic form $x_1^2+\dots+x_r^2$ is not universal over $K$ for any $r\in\Z^+$.	
\end{theorem}

We say that a quadratic form $Q$ with the Gram matrix $M_Q\in \Z^{r\times r}$ is \defined{represented} (over $\Z$) by a quadratic form $Q'$ with the Gram matrix $M_{Q'}\in \Z^{s\times s}$ if there exists $X\in \Z^{s\times r}$ such that $M_Q=X^T M_{Q'}X$ (see \cite[$\mathsection 41B$]{O1}). Note that if $Q$ is unary, that is $r=1$ and $Q(x)=\alpha x^2$, say, then any quadratic form $Q'$ that represents $\alpha$ represents $Q$. In particular, this definition generalizes the notion of a quadratic form representing an integer to a quadratic form representing another quadratic form. (One can analogously define representation of forms by forms over $\co$.)

Two quadratic forms $Q$ and $Q'$ of the same rank are \defined{equivalent} if they represent each other, in other words, if $M_Q=X^T M_{Q'}X$ for an invertible matrix $X$.

\medskip

Let $Q_1(x_1,\dots, x_r), Q_2(y_1,\dots, y_s)$ be two quadratic forms (over $\Z$ or $\co$). Their \defined{orthogonal sum} is defined as the $(r+s)$-ary form $(Q_1\perp Q_2)(x_1,\dots, x_r,y_1,\dots, y_s):=Q_1(x_1,\dots, x_r)+ Q_2(y_1,\dots, y_s)$. In the language of lattices, if $(L_1, Q_1)$ and $(L_2, Q_2)$ are the corresponding quadratic lattices, then their orthogonal sum is the lattice $L_1\oplus L_2$ equipped with the quadratic form $Q_1\perp Q_2$.
Similarly when $Q$ is a quadratic form of rank $r$ and $m\in\N$, then $Q^{\perp m}:=Q\perp Q\perp\dots\perp Q$ ($m$-times) is a quadratic form of rank $rm$.

Let $Q$ be a $\Z$-form. If it cannot decomposed as the orthogonal sum $Q=Q_1\perp Q_2$ of $\Z$-forms $Q_1, Q_2$, we say that $Q$ is an \defined{indecomposable} form, otherwise that it is \defined{decomposable}. Each $\Z$-form $Q$ can be uniquely decomposed as the orthogonal sum of indecomposable $\Z$-forms, its \defined{indecomposable constituents}; we also have analogous notions for totally positive quadratic forms over $\co$. 

\section{Pythagoras number}\label{s:3}

 Let us start with a preliminary consideration of the rank of a $\Z$-form that represents a given totally positive element $\alpha$ of an order $\co\subset K$.

\begin{prop}\label{lemma:3}
	If $\alpha\in\co^{+}$ is represented by some $\Z$-form over $\co$, then there exists a $\Z$-form $Q$ of rank at most $d$ that represents $\alpha$ over $\co$.
	Moreover, there exists a $d$-ary positive semidefinite quadratic form $Q_0$ with $\Z$-coefficients that represents $\alpha$ over $\co$.
\end{prop}
\begin{proof}
	Let $\omega_1,\ldots,\omega_d$ be an integral basis of the order $\co$ and let $Q'$ be a $\Z$-form of rank $r$ that represents $\alpha$, i.e., $Q'(v)=\alpha$ for some $v\in \co^r$. We can write $v=\sum^d_{i=1}v^{(i)}\omega_i,$
	where $v^{(i)}\in \Z^r$. In particular, we have that 
	\[Q'(v)=Q'\left(\sum^d_{i=1}v^{(i)}\omega_i\right)=\sum^d_{i=1}\omega^2_iQ'(v^{(i)})+2\sum_{i>j}\omega_i\omega_jB_{Q'}(v^{(i)},v^{(j)}).\]
	
	Consider the quadratic form $Q_0$ with the Gram matrix $(B_{Q'}(v^{(i)},v^{(j)}))\in\Z^{d\times d}$, i.e., 
	\[
	Q_0(x_1, \dots, x_d)=\sum_{1\leq i,j\leq d}B_{Q'}(v^{(i)},v^{(j)})x_ix_j=Q'\left(\sum^d_{i=1}v^{(i)}x_i\right).
	\]
	This quadratic form is positive semidefinite, as $Q'$ is a positive definite quadratic form. The form $Q_0$ is $d$-ary and we see that $Q_0(\omega_1,\ldots,\omega_d)=\alpha$. 
	
	Note that if $M_{Q'}$ is the $r\times r$ Gram matrix of $Q'$ and $V=(v^{(1)},\dots, v^{(d)})\in\Z^{r\times d}$ is the $r\times d$ matrix containing the vectors $v^{(i)}$ as columns, then for the $d
	\times d$ Gram matrix $M_{Q_0}$ of $Q_0$ we have 
	\begin{equation}\label{eq:d-ary}
		M_{Q_0}=V^TM_{Q'}V, 
	\end{equation}
	i.e., the quadratic form $Q'$ represents the form $Q_0$.
	
	To finish the proof by proving the existence of a $\Z$-form $Q$ that represents $\alpha$, let us prove the following lemma (that is essentially well-known, e.g., see the discussion at the beginning of \cite{Mo2}).
	
\begin{lemma*}
Let $Q_0(x)=\sum_{i=1}^d a_{ij}x_ix_j$ be a $d$-ary positive semidefinite quadratic form with $\Z$-coefficients. Then there is a $\Z$-form $Q$ of rank $\leq d$ such that $Q$ represents $Q_0$.
\end{lemma*}
	
\begin{proof}
	The \defined{radical} $\mathrm{Rad}(L):=\{v \in L\mid B_{Q_0}(v,w)=0\mbox{ for all } w\in L\}$ of the lattice $(L=\mathbb{Z}^d,Q_0)$ is clearly a sublattice of $L$. Further, it makes sense to define the quotient lattice $(\mathbb{Z}^d/\mathrm{Rad}(L),Q_1)$ by $Q_1(v\mod \mathrm{Rad}(L)):=Q_0(v)$. By the classification of finitely generated abelian groups, we have that $\mathbb{Z}^d/\mathrm{Rad}(L)$ is isomorphic to $\mathbb{Z}^m$ for some $m\le d$. 
	
	The quadratic lattice $(\Z^m, Q)\simeq (\mathbb{Z}^d/\mathrm{Rad}(L),Q_1)$ is then positive definite (by the Cauchy-Schwarz inequality, if $Q(v)=0$, then $B(v,w)=0$ for all $w$), i.e., $Q$ is a $\Z$-form of rank $m\le d$ which represents $Q_0$ by definition.	
\end{proof}	

As the form $Q$ from the lemma represents $Q_0$ (over $\Z$), it clearly represents all the elements that are represented by $Q_0$, not only over $\Z$, but also over $\co\supset \Z$.
In particular, $Q$ represents $\alpha$, which finishes the proof of Proposition \ref{lemma:3}.
\end{proof}

This proposition in particular gives an algorithm for deciding whether a given element $\alpha\in\co$ is represented by some $\Z$-form. Namely, it allows us to restrict our attention only to $\Z$-forms $Q$ of rank at most $d$.
Now we can use the following theorem of Conway-Sloane.

\begin{theorem}\label{thm:cs}\cite[Theorem 1]{CS1}
	For each $r\in\Z^+$ there is $k=k(r)\in\Z^+$ with the following property: If $Q$ is a classical $\Z$-form of rank $r$, then the form $kQ$ is the sum of squares of linear forms, in other words, it is represented by the quadratic form $x_1^2+\dots+x_s^2$ for some $s$. For $r\leq 5$ we have $k(r)=1$.	
\end{theorem}

Returning to the situation when a given element $\alpha$ is represented by a $\Z$-form $Q$ of rank at most $d$, we see that $2Q$ is classical and so
$2kQ$ is represented by the sum of squares (where $k=k(d)$). Thus we need only to check whether $2k\alpha$ is the sum of squares in $\co$ (which we can do, e.g., using Lemma~\ref{lemma:norm_bound}a)) and then whether these squares are of the correct shape corresponding to the decomposition of $2kQ$.

\

For a ring $R$, let $\mathrm{Sq}(R)$ denote the set of elements that are sums of squares in $R$ and $\mathrm{Sq}_m(R)$ denote the set of elements that are sums of $m$ squares in $R$. Then 
\[\mathrm{P}(R):=\min\{m:\mathrm{Sq}_m(R)=\mathrm{Sq}(R)\}\] 
is the \defined{Pythagoras number} of $R$ (if no such $m$ exists, then $\py(R):=\infty$). For a totally real number field $K$, we have that $\mathrm{P}(K)\le 4$ \cite{Si2,Ho}.  For orders $\co$ in $K$, $\mathrm{P}(\co)$ is finite, but can grow arbitrary large \cite{Sc2}. In this work, Scharlau also asked whether the Pythagoras number of orders is bounded in terms of the degree of the field $K$. 
Let us now show that this is indeed so.

\begin{cor}\label{cor:pythagoras}
	Let $\co$ be an order in a totally real number field of degree $d$. Then  $\mathrm{P}(\co)\le f(d)$, where $f$ is some function which depends only on $d$.
	If  $d=2,3,4$, or $5$, then $\mathrm{P}(\co)\le d+3$.
\end{cor}

\begin{proof}
	Assume that $\alpha\in\mathrm{Sq}(\co)$, i.e., $\alpha$ is represented by the $\Z$-form $x_1^2+\dots+x_{r}^2$ over $\co$ for some $r$.
	From Proposition \ref{lemma:3} it follows that $\alpha$ is represented by some $d$-ary semidefinite form $Q_0$. Moreover, from its construction \eqref{eq:d-ary} in the proof of Proposition \ref{lemma:3}, we see that this form $Q_0$ is represented by the original form $x_1^2+\dots+x_{r}^2$.
	
	Icaza \cite[Proposition 3]{Ic} (cf. also \cite[Theorem 1.3]{BI}) showed that \emph{there exists a function $f(d)$, which depends only on $d$, such that any quadratic form of rank $d$ that is represented by the sum of squares quadratic form is represented by the sum of $f(d)$ squares}. Therefore, $Q_0$ is represented by $f(d)$ squares $y_1^2+\cdots+y_{f(d)}^2$.
	Since $Q_0$ represents $\alpha$ over $\co$, we see that $\alpha$ is the sum of $f(d)$ squares of elements of $\co$. In other words, 
	$\mathrm{P}(\co)\le f(d)$. 
	
	Finally, when $2\leq d\leq 5$, then we are dealing with $Q_0$ of rank $d\leq 5$. For such forms, Ko and Mordell \cite{Ko,Mo} proved that they are represented by at most $d+3$ squares. Therefore we can take $f(d)=d+3$ for $2\leq d\leq 5$.
\end{proof}

The bound for real quadratic number fields $K=\Q(\sqrt n)$ is sharp \cite{Pe}. In fact, one can show that  $\mathrm{P}(\co_K)=3$ for $n=2,3,5$ \cite{Co,Sc1} and determine all $n$ for which $\mathrm{P}(\co_K)=4$ (as in \cite{CP}).

Let us now consider $\Z$-forms again. If we restrict to classical ones, we get the following result:

\begin{cor}\label{cor:univ class}
	Let $K\not=\Q(\sqrt{5})$ be a totally real number field of degree $d>1$ and $\co_K$ the ring of integers in $K$. Then there does not exist a classical $\Z$-form of rank $3, 4,$ or $5$ that is universal over $\co_K$. 
\end{cor}

\begin{proof}
	For contradiction assume that $Q$ is a classical $\Z$-form of rank strictly less than $6$ that is universal over $\co_K$. By Theorem \ref{thm:cs} of Conway-Sloane, $Q$ is represented by the sum of $s$ squares $x_1^2+\dots+x_s^2$ for some $s$. Since $Q$ is universal over $\co_K$, it follows that the sum of squares $x_1^2+\dots+x_s^2$ is universal as well.
	But this is impossible if $K\not=\Q(\sqrt{5})$ by Siegel's Theorem \ref{thm:siegel}.
\end{proof}

The previous corollary answers a (very) special case of Kitaoka's conjecture \cite{Km} that there exist only finitely many totally real number fields which admit a universal ternary quadratic form.
Note that the use of Theorem \ref{thm:siegel} was the only place in the proof of Corollary \ref{cor:univ class} where we used the assumption that $\co_K$ is the maximal order. This probably can be avoided by generalizing Siegel's theorem to general orders (most likely using essentially the same proof).

\section{Forms of $E$-type}\label{s:4}

Given two (positive definite) $\Z$-lattices $(L_1,{Q_1})$ and $(L_2,Q_2)$, we define their tensor product over $\Z$ as $(L_1\otimes L_2,Q_1\otimes Q_2)$, so that 
\[(Q_1\otimes Q_2)(v\otimes w)=Q_1(v)Q_2(w)\mathrm{\ and\ }B_{Q_1\otimes Q_2}(v\otimes v',w\otimes w')=B_{Q_1}(v,w)B_{Q_2}(v',w')\]
for all $v,v'\in L_1$ and all $w,w'\in L_2$. Given bases $\{v^{(1)},\ldots,v^{(b)}\}$ of $L_1$ and $\{w^{(1)},\ldots,w^{(c)}\}$ of $L_2$, then $\{v^{(i)}\otimes w^{(j)}\}$ is the canonical basis of the tensor product $L_1\otimes L_2$. The Gram matrix $M_{Q_1\otimes Q_2}$ of $L_1\otimes L_2$ with respect to this canonical basis equals $M_{Q_1}\otimes M_{Q_2}$, i.e., the Kronecker product of the Gram matrices of $Q_1$ and $Q_2$ (see \cite[Chapter 7]{K1} for more details on tensor products of lattices).

In general $(L_1\otimes L_2,Q_1\otimes Q_2)$ is not a $\Z$-lattice, as the quadratic form $Q_1\otimes Q_2$ need not be integer valued: for example, the tensor product of the lattice corresponding to the non-classical quadratic form $x^2+xy+y^2$ with itself will have quadratic form 
$$X^2+Y^2+Z^2+W^2+XY+XZ+YZ+YW+\frac12 (XW+YZ).$$
However, this happens if and only if we tensor two non-classical forms; as long as one of the forms is classical, the tensor product will be a $\Z$-lattice. This will always be the case in our paper.

\

For a $\Z$-lattice $(L,Q)$, let $\min(L):=\min_{0\not=v\in L}Q(v)$ be the minimum of $L$ and
\[\mathcal{M}(L):=\{v\in L:Q(v)=\min(L) \}\]
be the set of minimal vectors of $L$.
For two $\Z$-lattices $(L_1,Q_1)$ and $(L_2,Q_2)$ (one of which is classical) we clearly have
\[\min(L_1\otimes L_2)\le\min(L_1)\min(L_2).\]
There are examples of this inequality being strict (see \cite[page~47]{MH}), but there are important classes of lattices for which one has equality:

\begin{definition}
	We say that a $\Z$-lattice $L$ is of \defined{$E$-type} if 
	$\mathcal{M}(L\otimes M)\subset\{ v\otimes w :v\in \mathcal{M}(L), w\in \mathcal{M}(M)\}$ for every classical $\Z$-lattice $M$.
\end{definition}

Note that although lattices of $E$-type are usually defined only for classical $\Z$-lattices in the literature (e.g., \cite{K1}), we are extending the definition also to non-classical lattices. Nevertheless, all the results concerning lattices of $E$-type, such as Kitaoka's Theorem \ref{rmk:E-type} below, still hold since a non-classical lattice $L$ is of $E$-type if and only if the classical lattice $2L$ is of $E$-type.

Given two $\Z$-lattices $L_1$ and $L_2$, then of course not all elements of $L_1\otimes L_2$ are \defined{split}, i.e., of the form $v_1\otimes v_2$, $v_i\in L_i$. However, if either of $L_i$ is of $E$-type, then all the minimal vectors of $L_1\otimes L_2$ are split \cite[Lemma 7.1.1]{K1}.

Although certainly not every lattice is of $E$-type, this is true in several important cases, as the following theorem of Kitaoka shows.

\begin{theorem}\label{rmk:E-type}\cite[Theorems 7.1.1, 7.1.2, 7.1.3]{K1}
	Let $Q$ be a $\Z$-form of rank $r$. Then $Q$ is of $E$-type if at least one of the following conditions holds:
	\begin{itemize}
		\item $r\leq 43$,
		\item $\min(Q)\le 3$,
		\item $Q(x_1,\dots,x_r)=\Tr_{K/\Q}((\sum x_i\omega_i)^2)$, where $K$ is an abelian number field of degree $r$ with integral basis $\omega_1,\dots,\omega_r$.
	\end{itemize} 
\end{theorem}

We will only use the first criterion in this paper; it is an open question what is the smallest rank of a  $\Z$-form not of $E$-type. 

\

From now on, we will work only with the maximal order $\co_K$ in a totally real number field $K$ of degree $d$, although as we discussed in the Introduction, probably many of our results generalize to the case of general orders.

For $\delta\in \co_K^{\vee,+}$, we can consider the ``twisted trace form", i.e., the quadratic form $T_{\delta}(x)=\Tr(\delta x^2)$ for $x\in\co_K$. 
Note that $T_{\delta}$ is not a quadratic form over $\co_K$ (as $T_\delta(\alpha x)\neq \alpha^2T_\delta(x)$ for general $\alpha\in\co_K$), but
$(\co_K,T_\delta)$ is a quadratic $\Z$-lattice of rank $d$.
Fixing an integral basis $\omega_1,\dots,\omega_d$ for $\co_K$, we identify $\co_K$ with $\Z^d$. Then we can
denote $t_\delta$ the $\Z$-form of rank $d$ such that 
$$t_{\delta}(x_1, \dots, x_d):=\Tr\left(\delta \left(\sum x_i\omega_i\right)^2\right),$$
 i.e., $(\co_K, T_{\delta})=(\Z^d, t_{\delta})$ under the identification of $\co_K$ with $\Z^d$. 

Since $\delta$ lies in the codifferent, the form $T_{\delta}$ is $\Z$-valued. Because $\delta$ is totally positive, all the values of $T_{\delta}$ are positive (except for $T_\delta(0)=0$), 
and so $t_{\delta}$ is positive definite. 
Moreover, we see that $B_{T_\delta}(\beta,\beta')=\Tr(\delta\beta\beta')$ for all $\beta,\beta' \in \co_K$, and so the Gram matrix of the form $t_\delta$ is $(\Tr(\delta\omega_i\omega_j))_{ij}$, hence all its entries are (rational) integers.
In other words, we have verified that $t_{\delta}$ is a classical $\Z$-form, and so it makes sense to consider the tensor product $t_\delta\otimes Q$ with any $\Z$-form $Q$.
Finally, although $t_{\delta}$ of course depends on the choice of the integral basis, we will not need to worry about this, as the basis will be considered fixed throughout the paper.
We now have the following classical result on tensor products.

\begin{lemma}\label{lemma:tensor} 
	For a $\Z$-form $Q$ of rank $r$ and $\delta\in\co_K^{\vee,+}$, we have that 
	\[(\co_K^r,\Tr(\delta Q))=(\co_K\otimes\Z^r,T_\delta\otimes Q)=(\Z^d\otimes\Z^r,t_\delta\otimes Q).\]
\end{lemma}

Hence we will freely switch between these three $\Z$-lattices (and the corresponding quadratic forms) in the following.
In particular, note that they have the same minimum.

\begin{proof}
	Given that $\co_K^r$, $\co_K\otimes\Z^r$, and $\Z^d\otimes\Z^r$ are isomorphic as $\Z$-modules and that $T_\delta$ and $t_\delta$ are equivalent under this isomorphism by definition, it suffices to show that $\Tr(\delta Q)$ is equivalent to $T_\delta\otimes Q$. It suffices to show that the corresponding bilinear forms are equal on all split vectors.
	
	Let $\beta,\beta' \in \co_K$, $w,w'\in \Z^r$, and $B_Q(x,y)=\sum_{i, j}a_{ij}x_iy_j$. Then
	\begin{align*}
	B_{T_\delta\otimes Q}(\beta\otimes w,\beta'\otimes w')&=B_{T_\delta}(\beta,\beta')B_Q(w,w')
	=\Tr(\delta \beta\beta')\sum_{i, j}a_{ij}w_iw_j'\\
	&=\Tr\left(\delta \sum_{i, j}\beta\beta' a_{ij}w_iw_j'\right)
	=\Tr(\delta B_Q(\beta w,\beta'w'))\\
	&=B_{\Tr(\delta Q)}(\beta w,\beta'w').\qedhere
	\end{align*}
\end{proof}

The most important case for us will be when $t_{\delta}$ is of $E$-type, which we will assume from now; let us summarize all our assumptions for the rest of the paper:
\begin{itemize}
	\item $K$ is a totally real number field of degree $d$ over $\Q$,
	\item $\co_K$ is the ring of integers in $K$,
	\item the quadratic form $t_{\delta}$ is of $E$-type for every $\delta\in \co_K^{\vee,+}$; this is true if $d\leq 43$ by Theorem \ref{rmk:E-type},
	\item $Q(x)$ is a $\Z$-form of rank $r$, i.e., a positive definite quadratic form with $\Z$-coefficients.	
\end{itemize}

Let us now prove a series of auxiliary results that restrict possible number fields $K$ over which there may exist a universal $\Z$-form.

\begin{prop}\label{prop:indecom}\label{cor:minforms} 
	Assume that an indecomposable element $\alpha\in \co_K^+$ is 
	represented by $Q$ over $\co_K$ and satisfies $\Tr(\delta\alpha)=\min(t_{\delta}\otimes Q)$ for some $\delta\in \co_K^{\vee,+}$. Then $\alpha$ is a square in $\co_K$ and $\min(Q)=1$.
\end{prop}

\begin{proof}
	From the assumption that $\Tr(\delta\alpha)=\min(t_{\delta}\otimes Q)$, we conclude that the element $v$ of $\co_K^r$ representing $\alpha\in\co_K$ is a minimal vector of the quadratic form $t_{\delta}\otimes Q$ (which we identify with $\Tr(\delta Q)$ by Lemma \ref{lemma:tensor}).	
	Since $t_\delta$ is of $E$-type, the minimal vector $v\in\co_K^r=\co_K\otimes\Z^r$ is split, that is, $v=\beta\otimes w$, where $\beta \in \co_K$ and $w\in \Z^r$. We then have
	\[
	\alpha=Q(v)=
	\beta^2Q(w).
	\]
	Given that $\alpha$ is indecomposable and $Q(w)\in\Z^+$, we conclude that $Q(w)=1$ (otherwise $\alpha=\beta^2+\beta^2(Q(w)-1)$) and that $\alpha$ is a square.  
\end{proof}

\begin{cor}\label{cor:units} 
	If $Q$ is universal over $\co_K$, then every totally positive unit is a square in $\co_K$. Hence there is a unit of every signature in $\co_K$.
\end{cor}

\begin{proof}
	By Lemma \ref{lemma:norm_bound}b) we know that totally positive units are indecomposable.
	Hence in order to be able to use Proposition \ref{prop:indecom}, it suffices to show that for a given totally positive unit $\varepsilon$ there exists $\delta\in \co_K^{\vee,+}$ such that $\Tr(\delta\varepsilon)=\min(t_{\delta}\otimes Q)$. 
	
	We are assuming that $Q$ is universal, and so let $w\in\co_K^r$ be such that $Q(w)=\varepsilon$. 
	Clearly $\varepsilon^{-1}\in\co_K^{\vee,+}$ and the minimum of $t_{\varepsilon^{-1}}$ is greater or equal than the minimal value of $\Tr$ on totally positive integers. This minimal value is $d=\Tr(1)$, as by the inequality between arithmetic and geometric means we have $\Tr(\alpha)\geq d\mathrm N(\alpha)^{1/d}\geq d$.
	 
	Our assumption that $t_{\varepsilon^{-1}}$ is of $E$-type then implies
	\begin{align*}
	\min(t_{\varepsilon^{-1}}\otimes Q)&=\min(t_{\varepsilon^{-1}})\min(Q)\ge d\min(Q)\ge d.
	\end{align*}
	On the other hand, 
	$$\Tr(\varepsilon^{-1} Q(w))=\Tr(1)=d,$$
	and so $\varepsilon$ satisfies the assumptions of Proposition \ref{prop:indecom} for $\delta=\varepsilon^{-1}$, hence it is a square.
	
	Finally, recall the well-known fact \cite[p. 111, Corollary 3]{N} that every totally positive unit is a square in $\co_K$ if and only if there is a unit of every signature in $\co_K$.
\end{proof}

\begin{lemma}\label{lemma:squares} Assume that $Q$ is universal over $\co_K$ and let $\alpha\in \co_K^+$. If there exists $\delta \in \co_K^{\vee,+}$ such that $\Tr(\delta\alpha)\le \Tr(\delta\beta)$ for all $\beta\in\co_K^+$, then $\alpha$ is a unit in $\co_K$. 
\end{lemma}

Of course, when $Q$ is universal, the condition $\Tr(\delta\alpha)\le \Tr(\delta\beta)$ (for all $\beta\in\co_K^+$) in the lemma is equivalent to our earlier assumption from Proposition \ref{prop:indecom} that $\Tr(\delta\alpha)=\min(t_{\delta}\otimes Q)$.

\begin{proof}
	Every element $\alpha$ that satisfies the assumption of the lemma is indecomposable (if $\alpha=\alpha_1+\alpha_2$ with $\alpha_1,\alpha_2\in\co_K^+$, then $\Tr(\delta\alpha)>\Tr(\delta\alpha_1)$), and Proposition \ref{prop:indecom} implies that $\alpha$ is a square, say $\alpha=\gamma^2$ for some $\gamma\in\co_K$.
	Suppose that $\alpha$ satisfies the assumption, but is not a unit. We can also assume without loss of generality that $\alpha$ is such an element with the smallest possible norm.
	
	Let $\varepsilon\in \co_K^{\times}$ be such that $\varepsilon\gamma\succ 0$; such a unit exists by Corollary \ref{cor:units}. Then $Q$ represents $\varepsilon\gamma$, so there is $v\in\co_K^r$ such that $Q(v)=\varepsilon\gamma$. Denote $\delta':=\delta\gamma\varepsilon^{-1}\in \co_K^{\vee,+}$. Then 
	\[\min(t_{\delta'}\otimes Q)\ge \min(t_{\delta}\otimes Q).\]
	
	The equality here is attained: Let $a$ be the element of $\Z^d\otimes\Z^r$ that corresponds to $v\in\co_K^r$ by Lemma \ref{lemma:tensor}. Then
	$$(t_{\delta'}\otimes Q)(a)=\Tr(\delta'Q(v))=\Tr(\delta'\varepsilon\gamma)=\Tr(\delta\alpha)=\min(t_{\delta}\otimes Q).$$ 
	
	Therefore the element $\varepsilon\gamma\in\co_K^+$ has the property that $\Tr(\delta'\varepsilon\gamma)\le \Tr(\delta'\beta)$ for all $\beta\in\co_K^+$, but has norm 
	$\mathrm{N}(\varepsilon\gamma)=|\mathrm{N}(\gamma)|=\sqrt{\mathrm{N}(\alpha)}<\mathrm{N}(\alpha)$. This is a contradiction with our assumption that $\alpha$ is the element of smallest possible norm. 
\end{proof}

Clearly, if for $\alpha\in\co_K^+$ there exists $\delta\in \co_K^{\vee,+}$ such that $\Tr(\delta\alpha)=\min(t_{\delta})$, then $\alpha$ is indecomposable. 
Unfortunately, the converse implication does not hold, one counterexample being the indecomposable $\zeta_7^2+\zeta^{-2}_7-2$ in $K=\QQ(\zeta_7+\zeta_7^{-1})$.
Nevertheless, it holds in real quadratic fields and there are similar characterizations in higher degrees as well \cite{KT}.

\

Let us now record for further use the following well-known facts.

\begin{lemma}\label{lem:classical}
	Let $(L,Q)$ be a classical $\Z$-lattice of rank $r$ such that $\min(Q)=1$. Then:
	
	a) The number $\# \mathcal M(L)$ of minimal vectors $v\in L$ is at most $2r$.
	
	b) If $Q$ is moreover indecomposable quadratic form, then $r=1$ and $Q$ is equivalent to the quadratic form $x^2$.
\end{lemma}

\begin{proof}
	We shall prove the lemma as an easy application of the following classical Theorem (e.g., \cite[Proposition 4.10.7]{M}):
	
	\emph{Let $(L,Q)$ be a classical $\Z$-lattice of rank $r$ such that $\min(Q)=1$. Let $L_1$ be the sublattice of $L$ generated by the minimal vectors $\mathcal M(L)$ of $L$ (i.e., by $v\in L$ such that $Q(v)=1$), and $Q_1$ the restriction of $Q$ to $L_1$. Then $L_1$ is equivalent to $\Z^s$ equipped with the usual inner product (for some $s\leq r$), and $L=L_1\perp L_2$ for some lattice $L_2$.}
	
	a) By definition, each $v\in L$ with $Q(v)=1$ corresponds to precisely one $v'\in L_1$ with $Q_1(v')=1$. The number of such vectors in $L_1$ is the same as the number of these vectors in $\Z^s$, which is clearly $2s\leq 2r$.
	
	b) Since $L=L_1\perp L_2$, the indecomposability assumption implies that $L=L_1$, which is equivalent to $\Z^s$. But this is indecomposable only when $s=r=1$.
\end{proof}

We can now turn our attention to the (non-)existence of universal $\Z$-forms using the results established above.
For specific fields, one can often use the following proposition to deal with classical forms, although in general there of course need not exist any non-unit with norm smaller than $2^d$.

\begin{prop}\label{prop:norm2}
	If there exists $\alpha\in \co_K$ such that $1<|\mathrm{N}(\alpha)|<2^d$, then there does not exist a classical $\Z$-form that is universal over $\co_K$. 
\end{prop}

\begin{proof}
	Let $\alpha \in\co_K$ be such that $1<|\mathrm{N}(\alpha)|<2^d$.  If a $\Z$-form $Q$ is universal over $\co_K$, then by Corollary \ref{cor:units} there are units of all signatures in $\co_K$. Thus, after multiplying by a suitable unit, we can assume that $\alpha\succ 0$, and furthermore take $\alpha$ to be such element of the smallest possible norm. By Lemma \ref{lemma:norm_bound}b), $\alpha$ is indecomposable.
	
	Let us consider the orthogonal decomposition of $Q$ (over $\Z$), say $Q=Q_1\perp\dots\perp Q_k$. If $v=(v^{(1)},\dots, v^{(k)})$ is the corresponding decomposition of a vector $v$ representing $\alpha$, then $\alpha=Q(v)=Q_1(v^{(1)})+\dots+Q_k(v^{(k)})$. We have $Q_i(v^{(i)})\succeq 0$ for all $i$, and so all but one of these values must be zero, as $\alpha$ is indecomposable.
	Hence $\alpha$ is represented by one of the indecomposable constituents, which we will denote $Q'$. Note that $Q'$ is an indecomposable classical $\Z$-form (because $Q$ was classical).

	Let $m\in \mathbb{N}$ be such that $m\Tr\left(\frac{\beta}{\alpha}\right)\in \Z$ for all $\beta \in \co_K$. We denote $\delta:=\frac{m}{\alpha}\in \co_K^{\vee,+}$. Let us now prove that $\alpha$ is a square by distinguishing two possible cases: 
	
	1. If $\Tr(\delta\alpha)=\min(t_{\delta}\otimes Q')$, then by Proposition \ref{prop:indecom} it follows that $\alpha$ is a square. 
	
	2. Otherwise there exists $\beta \in \co_K^{+}$ such that $\Tr(\delta\alpha)>\Tr(\delta\beta)$ and $Q'(w)=\beta$ with $w\in \mathcal{M}(t_{\delta}\otimes Q')$. Therefore
	$$
	dm=\Tr(m)=\Tr(\delta\alpha)>\Tr(\delta\beta)=m\Tr\left(\frac{\beta}{\alpha}\right),
	$$
	and so $d>\Tr\left(\frac{\beta}{\alpha}\right)$. 
	
	The inequality between arithmetic and geometric means then gives
	$$1	>\frac 1d \Tr\left(\frac{\beta}{\alpha}\right)\geq\mathrm{N}\left(\frac{\beta}{\alpha}\right)^{1/d},$$ and so $2^d>\mathrm{N}(\alpha)	> \mathrm{N}(\beta)$.
 	Thus $\beta$ is indecomposable by Lemma \ref{lemma:norm_bound}b).

	Applying Proposition \ref{cor:minforms} for the element $\beta$ we obtain that $\min(Q')=1$, and given that $Q'$ is an indecomposable classical $\Z$-form, this implies that $Q'$ is just a form of one variable, $Q'(x)=x^2$, by Lemma \ref{lem:classical}b). This form represents $\alpha$, which therefore is a square. 
	
	\medskip
	
	In both cases we proved that $\alpha$ is a square, say $\alpha=\gamma^2$, and therefore there exists a non-unit element $\gamma$ with a smaller norm than $\alpha$, contradicting the assumption of minimality of norm of $\alpha$. 	 
\end{proof}

\theoremstyle{plain}
\newtheorem*{thm:1}{Theorem \ref{Thm:1}}
\begin{thm:1}
	There does not exist a $\Z$-form that is universal over a real quadratic number field $K$, unless $K=\Q(\sqrt{5}).$
\end{thm:1} 

\begin{proof}
	Let $K=\Q(\sqrt D),$ where $D\ge 2$ is squarefree integer. Let $\co_K=\Z[\omega]$ be the ring of integers in $K$, where $\omega=\sqrt{D}$ or $\frac{1+\sqrt{D}}{2}$, depending on whether $D\equiv 2, 3$, or $1\pmod{4}$. Let $f(x)$ be the minimal polynomial of $\omega$. From Corollary \ref{cor:units} it follows that if there exists a universal $\Z$-form over $\co_K$, then $\co_K$ has units of all signatures; fix a unit $\varepsilon>0$ such that its conjugate $\overline\varepsilon<0$. 
	
	We know that $\co_K^{\vee}=\frac 1{f'(\omega)}\co_K$. Thus we can let
	$\delta:=\varepsilon\cdot\frac 1{f'(\omega)}$ so that we have $\delta\succ 0$ and $\co_K^{\vee}=\delta\co_K$. The form $t_\delta$ has rank 2, and so from Kitaoka's Theorem \ref{rmk:E-type} it follows that $t_\delta$ is of $E$-type. 
	
	The element $\alpha:=\omega\varepsilon^{-1}\in\co_K^+$ satisfies $\Tr(\delta\alpha)=1$: Let us show the easy computation only in the case $\omega=\sqrt D$, when we have $\delta=\frac \varepsilon{f'(\omega)}=\frac \varepsilon{2\sqrt D}$, and so
	$\Tr(\delta\alpha)=\Tr\left(\frac {\sqrt D}{2\sqrt D}\right)=1$.
	(In fact, all elements with $\Tr(\delta\alpha)=1$ form an arithmetic progression \cite[Example 1]{Ya}.)
	
	By Lemma \ref{lemma:squares} we therefore have that $\alpha\in\co_K^+$ is a unit, and so $\mathrm{N}(\omega)=-\mathrm{N}(\alpha)=-1$.
	If $D\equiv 2,3 \pmod{4}$, then $\mathrm{N}(\omega)=\mathrm{N}(\sqrt{D})=-D\neq -1$. When $D\equiv 1\pmod{4}$, we have $-1=\mathrm{N}(\omega)=\frac{1-D}4$, and so the only possibility is $D=5$.
\end{proof}

\section{Dedekind zeta function}\label{s:5}

Let $K$ be a totally real number field of degree $d$ and let $\Delta_K$ denote the discriminant of $K$. 
Results of the previous section suggest that elements of the codifferent which have small trace play a key role in the study of universal $\Z$-forms; 
we shall use Siegel's formula \cite{Si1} to estimate the number of these elements in terms of Dedekind zeta function $\zeta_K(s)$. We start by reviewing its basic properties following \cite[\S 1]{Z}
as reference for all the facts that we mention.

Dedekind zeta function $\zeta_K(s)$ of $K$ for $s\in \mathbb{C}$
is the meromorphic function that for $\mathrm{Re}(s)>1$ satisfies
\[
\zeta_K(s)=\sum^{\infty}_{n=1}\dfrac{F(n)}{n^s},
\]
where $F(n)$ is the number of ideals in $\co_K$ of norm $n$ (and the norm of an ideal $I$ is $\mathrm N (I)=\# \co_K/I$).

We will be interested in the values at the points $s=2$ and $s=-1$; clearly $\zeta_K(2)>\dfrac{F(1)}{1^2}=1$. From the functional equation 
we see that
\begin{equation}\label{eq:zeta-fun}
\zeta_K(-1)=(-1)^d|\Delta_K|^{3/2}\left(\dfrac{1}{2\pi^2}\right)^d\zeta_K(2).
\end{equation} 

Assume from now on that the degree $d=2,3,4,5,7$, and let $b_d=\dfrac{1}{240},\dfrac{-1}{504},\dfrac{1}{480},\dfrac{-1}{264},\dfrac{-1}{24},$ respectively. 
Then we have
\begin{equation}\label{eq:zeta-neg}
\zeta_K(-1)=2^db_d\sum_{\substack{\alpha \in \co_K^{\vee,+}\\\Tr(\alpha)=1}}\sigma((\alpha)(\co_K^{\vee})^{-1}),
\end{equation} 
where
\begin{equation*}
\sigma(I)=\sum_{J\mid I}\mathrm N(J)
\end{equation*} 
and $(\alpha)$ denotes the fractional ideal $\alpha\co_K$.

Putting together \eqref{eq:zeta-fun} and \eqref{eq:zeta-neg}, we get
\begin{equation}\label{eq:zeta-sum}
\sum_{\substack{\alpha \in \co_K^{\vee,+}\\\Tr(\alpha)=1}}\sigma((\alpha)(\co_K^{\vee})^{-1})=\dfrac{(-1)^d}{b_d}|\Delta_K|^{3/2}\left(\dfrac{1}{4\pi^2}\right)^d\zeta_K(2).
\end{equation}

We can now use this formula to prove our main result that greatly restricts possible number fields of small degrees with a universal $\Z$-form.

\begin{theorem}\label{thm 16}
	There does not exist a totally real number field $K$ of degree $2,3,4,5,$ or $7$, with a principal codifferent ideal and a universal $\Z$-form defined over it, unless $K=\Q(\sqrt{5})$ or $\QQ(\zeta_7+\zeta_7^{-1})$.
\end{theorem}
\begin{proof}
	Since $d\leq 7<43$, by Kitaoka's Theorem \ref{rmk:E-type} the form $t_\delta$ is of $E$-type. Hence we can use the results of Section \ref{s:4}.
	Let us assume that there exists a universal $\Z$-form over $\co_K$. 
	By Corollary \ref{cor:units} there are units of all signatures in $\co_K$. By the assumption that $\co_K^{\vee}$ is a principal ideal, there exists some $\delta \in K$ such that $\co_K^{\vee}=(\delta).$ Without loss of generality, let $\delta\succ 0$. 
	
	By Lemma \ref{lemma:squares}, if $\alpha=\alpha'\delta\in\co_K^{\vee,+}$ is such that $\Tr(\alpha)=1$, then $\alpha'\in\co_K^{\times,+}$. As then $(\alpha)=(\delta)=\co_K^\vee$, we deduce that 
	\begin{equation*}
	\sigma((\alpha)(\co_K^{\vee})^{-1})=\sigma(\co_K^{\vee}(\co_K^{\vee})^{-1})=\sigma(\co_K)=1.
	\end{equation*}	
	Therefore, the left-hand side of \eqref{eq:zeta-sum} is equal to the number of $\alpha\in\co_K^{\vee,+}$ that have $\Tr(\alpha)=1$. The right-hand side of \eqref{eq:zeta-sum} is non-zero as $\zeta_K(2)\not=0$, which in particular implies that there is at least one such $\alpha\in\co_K^{\vee,+}$ and that $\min(t_{\delta})=1$. 
	
	Let $\alpha=\alpha'\delta\in\co_K^{\vee,+}$ be such that $\Tr(\alpha)=1$. We have already seen that by Lemma \ref{lemma:squares} we have $\alpha'\in\co_K^{\times,+}$, and moreover by Corollary \ref{cor:units}, $\alpha'$ is a square, so let $\alpha'=\beta^2$ for some $\beta\in\co_K$. Further let $v\in\Z^d$ be such that $\beta=\sum v_i\omega_i$ for an integral basis $\omega_i$ of $\co_K$. We then have 
	\[t_{\delta}(\pm v)=\Tr \left(\delta \left(\pm\sum v_i\omega_i\right)^2\right)=\Tr (\delta\beta^2)=\Tr (\delta\alpha')=1.
	\]
	Thus the two vectors $\pm v$ that were attached to $\alpha$ are minimal vectors of the $\Z$-lattice $(\Z^d,t_\delta)$, whose minimum is 1.
	In the discussion just before Lemma \ref{lemma:tensor} we saw that $t_{\delta}(x)$ is classical, and so it has at most $2d$ minimal vectors by Lemma \ref{lem:classical}a). Therefore,	
	\begin{equation*}
	2d\ge 2\cdot\#\{\alpha \in \co_K^{\vee,+}|\Tr(\alpha)=1\}=2\sum_{\substack{\alpha \in \co_K^{\vee,+}\\ \Tr(\alpha)=1}}\sigma((\alpha)(\co_K^{\vee})^{-1}).
	\end{equation*}
	
	Using this inequality and the fact that $\zeta_K(2)>1$ in equation \eqref{eq:zeta-sum}, we get
	\[ |\Delta_K|<\left|\left(4\pi^2\right)^d d b_d\right|^{2/3}. \]
	In the table below we summarize the resulting bounds:
	\[
	\begin{tabular}{| l | r| }
	\hline
	$d$ & $|\Delta_K|<$ \\
	\hline
	\hline
	2 & 5.6... \\
	3 & 51.2...\\
	4 & 742.8... \\
	5 & 14886.9... \\
	7 & 12386158.6...\\
	\hline
	\end{tabular}
	\]
	From online database of number fields (described in \cite{JR}; cf. also \cite{V}) we find that there are only a few totally real number fields $K=\QQ(\omega)$ that satisfy the above bounds:	
	\[
	\begin{tabular}{|c|l|c|c|}
	\hline
	$d$ &Minimal polynomial of $\omega$ & $|\Delta_K|$&Narrow class number\\
	\hline
	\hline
	2&$x^2 - x - 1$ &5&1\\
	3&$x^3 - x^2 - 2x + 1$&49&1\\
	4&$x^4 - x^3 - 3x^2 + x + 1$&725&1 \\
	5&$x^5 - x^4 - 4x^3 + 3x^2 + 3x - 1$&14641&1\\
	\hline
	\end{tabular}
	\]
	First, note that there does not exist a number field satisfying the above bound for degree 7. Quadratic, cubic, and quintic number fields correspond to the maximal real subfields of cyclotomic fields. In particular, we have $\QQ(\zeta_5+\zeta_5^{-1}), \QQ(\zeta_7+\zeta_7^{-1})$, and $\QQ(\zeta_{11}+\zeta_{11}^{-1})$, respectively. We need to exclude the cases in degrees $d=4$ and 5.
	
	\
	
	${d=4}.$
	Let $K=\Q(\omega)$, where $\omega$ is a root of $f(x)=x^4 - x^3 - 3x^2 + x + 1$. Given that $\Delta_f=\Delta_K$ (where $\Delta_f$ is the discriminant of the polynomial $f$), we have $\co_K=\Z[\omega]$. The totally positive integer $\omega+2$ has norm $11<2^4$, and so it is indecomposable by Lemma \ref{lemma:norm_bound}b), and from Proposition \ref{prop:norm2} it follows that there is no universal classical $\Z$-form. However, we also need to exclude non-classical forms.
	
	Assume that there is a universal (non-classical) $\Z$-form over $\co_K$. In particular it represents $\omega+2$, and so by Proposition \ref{lemma:3}, this element is represented by a $\Z$-form $Q'$ of rank $\leq 4$. But then $2(\omega+2)$ is represented by the classical $\Z$-form $2Q'$, which in turn is represented as sum of squares by Theorem \ref{thm:cs}.
	
	Hence it suffices to show that $2(\omega+2)$ cannot be represented as a sum of squares; assume that $2(\omega+2)\succeq \alpha^2$ for some $\alpha\neq 0$, without loss of generality $\alpha\succ 0$ (as there are units of all signatures).
	$2(\omega+2)$ is not a square, and so $2(\omega+2)= \alpha^2+\gamma$ with $\gamma\succ 0$.
	Then by Lemma \ref{lemma:norm_bound}a) we have $$\mathrm{N}(\alpha)^{1/2}+1\leq\mathrm{N}(\alpha)^{1/2} +\mathrm{N}(\gamma)^{1/4}\leq \mathrm{N}(2(\omega+2))^{\frac{1}{4}}=2\cdot 11^{1/4},$$ and so 
	$\mathrm{N}(\alpha)<7$. We can easily check in Magma that the only elements of norm less than 11 in $\co_K$ are units, and so if  $2(\omega+2)$ is a sum of squares, then it is a sum of squares of units. From the bound of Lemma \ref{lemma:norm_bound}a) it further follows that there are at most three summands. The trace of $2(\omega+2)$ is $18$, thus by checking all the combinations of totally positive units of small trace, we confirm that $2(\omega+2)$ cannot be represented as a sum of squares. 
	
	\
	
	$d=5$.
	Let $\omega=\zeta_{11}+\zeta_{11}^{-1}$ and $K=\QQ(\omega)$. There (again) exists a totally positive integer of norm 11 in $\co_K$, i.e., $\beta=\omega+2$.  Given that $\mathrm{N}(\beta)=11<2^5$, Lemma \ref{lemma:norm_bound}b) implies that $\beta$ is indecomposable. 
	
	By the different theorem \cite[Theorem 4.24]{N}, every ramified prime ideal of $\co_K$ divides the different $(\co_K^{\vee})^{-1}$. In our case, the norm of $\beta$ is prime and $\beta\mid 11\mid \Delta_K$, and so the principal ideal $(\beta)$ is a ramified prime ideal. It therefore divides the different, and so inversely, $\beta^{-1}$ lies in the codifferent, i.e., $\beta^{-1}\in\co_K^{\vee,+}$. 
	
	If there exists a universal $\Z$-form $Q$ over $K$, then 
	$$\min(t_{\beta^{-1}})\min(Q)=\min(t_{\beta^{-1}}\otimes Q)	\le \Tr(\beta^{-1}\beta)=\Tr(1)= 5,$$
	as $Q$ represents $\beta$. Since $\Z[\omega]=\co_K$, we compute in Magma the Gram matrix of $t_{\beta^{-1}}$ corresponding to the integral basis $\{1,\omega,\omega^2,\omega^3,\omega^4\}$:
	\begin{equation*}
	\begin{pmatrix}
	5&-5&11&-13&30\\
	-5&11&-13&30&-35\\
	11&-13&30&-35&86\\
	-13&30&-35&86&-94\\
	30&-35&86&-94&252
	\end{pmatrix}.\end{equation*}
	Using this, we check (in Magma again) that 
	$\min(t_{\beta^{-1}})=5$. Thus by Proposition \ref{prop:indecom} it follows that $\beta$ is a square, which is impossible given that the norm of $\beta$ is 11. Therefore, there does not exist universal $\Z$-form over $\co_K$.
\end{proof}

We note that the argument in the second part of the proof fails for $\QQ(\zeta_7+\zeta_7^{-1})$. Even though we can find an indecomposable element $\beta$ of norm $7$, we have that $\min(t_{\beta^{-1}})=2$, and so $\beta$ does not correspond to a minimal vector.

\section{Existence of universal forms}\label{s:6}

In the previous sections we have proved that there does not exist a universal $\Z$-form in a number of cases. Let us now turn our attention to the converse problem, namely, of proving the existence of a universal $\Z$-form over $\QQ(\zeta_7+\zeta_7^{-1})$.

\begin{lemma}\label{prop:7ind}
	Let $K=\QQ(\zeta_7+\zeta_7^{-1})$. Then the binary quadratic form $Q(x,y)=x^2+xy+y^2$ represents all indecomposable integers of $K$. 
\end{lemma}
\begin{proof}
	All totally positive units are squares in $\co_K$, and so $Q$ represents all of them, as $Q$ represents~1. 
	
	Let us now find all the indecomposable elements in $\co_K^+$. Brunotte \cite[Korollar 2]{Br2} proved that the norm of every indecomposable integer in a general real cyclotomic field $\Q(\zeta_n+\zeta_n^{-1})$ is at most $c_n^{\varphi(n)}$, where 
	$\varphi(n)$ is Euler's totient function and $c_n$ is certain constant defined using the unit group of $\Q(\zeta_n)$ (this is the constant denoted $c_{\Q^{(n)}}$ in \cite{Br} and $c(E_n)$ in \cite{Br2}).
	When $n=7$, Brunotte \cite[Table~2]{Br} also showed that $c_7<1.6$, and so each indecomposable has norm at most 16.
	We then check in Magma all these candidate elements to show that the only non-unit indecomposable integer of $K$ (up to multiplication by totally positive units) is an element of norm $7$, which can be written as \[(\zeta_7^2+\zeta_7^{-2}+2)+(\zeta_7+\zeta_7^{-1})+1.\] 
	
	This element is represented by $Q$, as we have that
	\begin{align*}
	(\zeta_7^2+\zeta_7^{-2}+2)+(\zeta_7+\zeta_7^{-1})+1&=(\zeta_7+\zeta_7^{-1})^2+(\zeta_7+\zeta_7^{-1})+1
	=x^2+xy+y^2,
	\end{align*}
	where $x= \zeta_7+\zeta_7^{-1}$ and $y=1$.
\end{proof}

Strictly speaking, we do not need this lemma for the following theorem, but it
gives us some important clues concerning the possible shape of a universal $\Z$-form over $\QQ(\zeta_7+\zeta_7^{-1})$. In particular, it may be advantageous for it to contain $x^2+xy+y^2$ as a subform.
This is indeed the case for the form $x^2+y^2+z^2+w^2+xy+xz+xw$, whose universality over $\QQ(\zeta_7+\zeta_7^{-1})$ we will now show. This form is the norm form for a subring of the quaternions, which was used by Deutsch \cite{De2} to prove its universality over ${\Q(\sqrt 5)}$; this form is also universal over $\Z$.

\begin{theorem}\label{prop:7univ}
	Let $K=\QQ(\zeta_7+\zeta_7^{-1})$. Then the quadratic form $Q(x,y,z,w)=x^2+y^2+z^2+w^2+xy+xz+xw$ is universal over $\co_K$. 
\end{theorem}
\begin{proof}
	We use genus theory (and the mass formula) in the proof, following similar lines as the proof of Theorem 3.1 in \cite{CKR}. Let us recall that two quadratic forms over a number field $K$ are in the same genus if they are equivalent at every completion of $K$. Therefore, if the genus of a quadratic form $Q$ consist of only one class, then the local-global principle holds for $Q$. Specifically, 
	in this case $Q$ represents all the elements $\alpha\in\co_K$ that it represents locally everywhere, i.e., all $\alpha$ that are represented over all the archimedean completions and over the rings of integers of all the non-archimedean completions.
	
	All the one-class genera of quadratic forms (of rank $\geq 3$) over all totally real fields $K\neq \Q$ were determined by Kirschmer \cite{K}; there are 471 such genera.
	Our form $Q=x^2+y^2+z^2+w^2+xy+xz+xw$ over $K=\QQ(\zeta_7+\zeta_7^{-1})$ is one of them, and so we know that the local-global principle holds for $Q$.
	Let us note that the classification uses Siegel's mass formula \cite{Si4} and that one can directly check in Magma that the class number of $Q$ over $K$ is 1 
	(the mass of $Q$ is $\frac{1}{1152}$ and the order of the automorphism group of $Q$ is $1152$). 
	
	Thanks to the local-global principle, it now suffices to show that $Q$ is universal locally.
	
	There are three archimedean places, all of them real, and $Q$ clearly represents all positive elements over each of them, as $Q$ is positive definite. 
	
	Over all the non-dyadic places, 2 is a unit, and so $Q$ is unimodular (i.e., the Gram matrix of $Q$ is invertible). 
	A well-known theorem \cite[92:1b]{O1} says that \emph{if $q$ is a unimodular quadratic form of rank $\geq 3$ over a non-dyadic local field $F$, then $q$ represents all the elements of the ring of integers $\co_F$.}
	As the rank of $Q$ is $4$, it follows that $Q$ is universal there too. 
	
	Finally, 2 is inert in $K$. We directly check that $Q$ represents all the square classes over the corresponding 2-adic completion of $K$, which is the degree $3$ extension of $\Q_2$ (and so contains $32$ square classes as we directly compute).
	
	Hence $Q$ is indeed universal over $\co_K$.
\end{proof}

From Hilbert reciprocity law it follows that there cannot exist a universal quadratic form over $\QQ(\zeta_7+\zeta_7^{-1})$ of smaller rank than 4  \cite[Lemma 3]{EK}. 

Let us remark that another quaternary quadratic form, $Q'=x^2+xy+y^2+z^2+zw+w^2$, that was also considered by Deutsch \cite{De}, appears to be universal over $\QQ(\zeta_7+\zeta_7^{-1})$ as well. However, this form has class number 2, and so we have not proved its universality. Deutsch proved the universality of $Q'$ over ${\Q(\sqrt 5)}$ and also showed that it is not universal over several real quadratic fields of small discriminant. By our Theorem \ref{Thm:1} it now follows that it is not universal over any other real quadratic field.

\

Finally, let us show a general proposition that provides a way of constructing a universal form from a quadratic form that represents all indecomposable integers.

\begin{prop}\label{prop:ind-univ}
	Let $K$ be a totally real number field of degree $d$ and $Q$ a totally positive quadratic form over $\co_K$ of rank $r$ that represents all indecomposable integers.
	Then there is $m\in\mathbb N$ such that the orthogonal sum
	$Q^{\perp m}$  is universal over $\co_K$. In fact, $m$ is the product of the Pythagoras number $\mathrm P(\co_K)$ with the number of indecomposable integers in $\co_K$ up to multiplication by squares of units.
\end{prop}

In particular, if there is a $\Z$-form that represents all indecomposables such as in Lemma \ref{prop:7ind}, then there is a universal $\Z$-form over $\co_K$. So if we were interested only in the existence of a universal $\Z$-form, we could have used this proposition instead of the specific (and much stronger!) construction of Theorem \ref{prop:7univ}.

\begin{proof}
	Let $B$ be a set containing certain indecomposable integers so that every indecomposable is of the form $\varepsilon^2\beta$ for some $\beta\in B$ and $\varepsilon\in\co_K^{\times}$. 
	Brunotte \cite{Br2} proved that every indecomposable has norm less than a constant $c_{\co_K}$ (depending only on $K$). Thus the set $B$ is finite.
	
	Every element of $\alpha\in\co_K^+$ can be expressed as the sum of indecomposables. By combining indecomposables $\beta\varepsilon_i^2$ with the same representative $\beta\in B$, we see that $\alpha$ can be written as a sum $\sum_{\beta\in B} \beta\sigma_\beta$, where each $\sigma_\beta$ is a sum of squares of units.
	The Pythagoras number of $\co_K$ is finite (cf. Corollary \ref{cor:pythagoras}), and so there is a constant $c=f(d)$ such that each $\sigma_\beta$ is a sum of $c$ squares, i.e., it is represented by the form 
	$x_1^2+\dots+x_c^2$.
	
	Each $\beta\in B$ is represented by the form $Q$, and so each of the elements $\beta\sigma_\beta=\beta(s_1^2+\dots+s_c^2)$ (for some $s_1,\dots, s_c\in\co_K$) is represented by $Q^{\perp c}$. The element $\alpha$ is the sum of $\# B$ such elements $\beta\sigma_\beta$, and so it is represented by the form 
	$\left(Q^{\perp c}\right)^{\perp (\# B)}=Q^{\perp (\# B\cdot c)}$.
\end{proof}

Let us conclude the article by the example of the form $Q=x^2+xy+y^2$ over $K=\QQ(\zeta_7+\zeta_7^{-1})$. In the proof of Lemma \ref{prop:7ind}, we have seen that, up to multiplication by units, there are exactly two indecomposables, 1 and $\beta=\zeta_7^2+\zeta_7^{-2}+\zeta_7+\zeta_7^{-1}+3$. As every totally positive unit in $\co_K$ is a square, we see that $\# B=2$.
We are in a number field of degree $d=3$, and so the Pythagoras number is $\leq 6$ by Corollary \ref{cor:pythagoras}. Thus the $\Z$-form $Q^{\perp 12}$ of rank 24 is universal over $K$.

Further, as pointed out by the anonymous referee, observing that $2\beta$ and $3\beta$ are sums of squares, this can be improved to show that the $\Z$-form $Q^{\perp 6}+x_1^2+\dots+x_6^2$ of rank 18 is universal over $K$.

\end{document}